\documentclass{article}

\usepackage[T1]{fontenc}
\usepackage{geometry}
\usepackage{enumitem}

\usepackage{amsmath,amssymb,amsthm}
\usepackage{tikz}
\usepackage{tikz-cd}
\usepackage{hyperref}
\usepackage{cancel}
\usepackage{thmtools}
\usepackage{thm-restate}

\usepackage[backend=bibtex,sortcites=true]{biblatex}

\newcommand{\di}{\mathrm{d}}
\newcommand{\wh}{\widehat}

\newcommand{\real}{\mathbb{R}}
\newcommand{\inte}{\mathbb{Z}}
\newcommand{\comp}{\mathbb{C}}
\newcommand{\octo}{\mathbb{O}}

\newcommand{\proj}{\mathbb{P}}

\newcommand{\so}{\mathfrak{so}}

\newcommand{\inv}{^{-1}}

\newcommand{\cspec}{C_{\mathrm{spec}}}
\newcommand{\ad}{\mathrm{ad}}
\newcommand{\rk}{\mathrm{rk}}
\newcommand{\vspan}{\mathrm{span}}
\newcommand{\calM}{\mathcal{M}}
\newcommand{\TD}{T}
\newcommand{\Spectrum}{\Sigma}

\DeclareMathOperator\Spin{Spin}

\DeclareMathOperator\SU{SU}
\DeclareMathOperator\Sp{Sp}

\DeclareMathOperator\vol{vol}

\DeclareMathOperator\End{End}

\DeclareMathOperator{\Tr}{Tr}
\DeclareMathOperator{\Cl}{Cl}

\newtheorem{theorem}{Theorem}
\newtheorem{lemma}[theorem]{Lemma}
\newtheorem{example}[theorem]{Example}
\newtheorem{proposition}[theorem]{Proposition}
\newtheorem{corollary}[theorem]{Corollary}

\theoremstyle{definition}
\newtheorem{definition}[theorem]{Definition}

\theoremstyle{remark}
\newtheorem{remark}[theorem]{Remark}

\setlist[enumerate, 1]{leftmargin = *, align = left}

\bibliography{zotero}
\title{Obstructions to Spin(7) Nahm transforms on tori}
\author{Spencer Whitehead \\ \small Duke University, \texttt{spencer.whitehead@duke.edu}}
\date{\today}

\begin{document}
\maketitle

\begin{abstract}
    The Nahm transform for 4-dimensional flat hyperk\"ahler tori is an isometry between the moduli space of anti-self-dual (ASD) instantons on a torus $T^4$ and the moduli space of ASD instantons on the dual torus $\wh{T^4}$ parametrising flat line bundles on $T^4$.
    This paper studies a generalised Nahm transform on an 8-dimensional torus with a $\Spin(7)$ structure.
    I construct instanton bundles with Dirac kernels respectively in positive and negative chiralities, demonstrating that the usual Nahm transform is not well-defined.
    I then define a notion of asymptotic holonomy for instantons twisted by a high power $k \gg 1$ of an instanton line bundle, and I show that this asymptotic holonomy reduces to $\Spin(7)$ to second order in $k$.
    Finally, I provide examples for which the asymptotic holonomy is $\mathfrak{u}(1)^{\oplus 4}$, and thus not $\Spin(7)$.
\end{abstract}

\section{Introduction}
\label{sec:introduction}

The Nahm transform for 4-dimensional tori is an isometry between the moduli space of gauge equivalence classes of anti-self-dual (ASD) instantons on a torus $T^4$ and the moduli space of ASD instantons on the dual torus $\wh{T^4}$.
The proof of the Nahm transform splits into three parts:
\begin{enumerate}[widest={\textbf{(Holonomy)}}]
    \item[\textbf{(Vanishing)}] For $E$ an irreducible ASD instanton on $T^4$ and $L_t, t \in \wh{T^4}$ a flat line bundle on $T^4$, the Dirac operator $D^+_t \colon \Gamma(S^+ \otimes E \otimes L_t) \to \Gamma(S^- \otimes E \otimes L_t)$ has zero kernel; thus $\wh{E}_t := \ker D^-_t \subset \Gamma(S^- \otimes E \otimes L_t)$ defines a bundle over $\wh{T^4}$ with a canonical connection $\wh{A}$.
    \item[\textbf{(Holonomy)}] The connection $\wh{A}$ is ASD.
    \item[\textbf{(Isometry)}] The Nahm transform is a duality; under the identification $\wh{\wh{T^4}} = T^4$, $E$ may be identified in a natural way with the Nahm transform of $\wh{E}$.
        The moduli spaces of ASD instantons over $T^4, \wh{T^4}$ have natural metrics, and the map $(E, A) \mapsto (\wh{E}, \wh{A})$ is an isometry.
\end{enumerate}
The statement \textbf{(Vanishing)} is an immediate consequence of the Lichnerowicz formula
\begin{equation}
    \label{eq:lichnerowicz}
    D_t^2  = \nabla_t^* \nabla_t + c(F_E)
\end{equation}
together with the algebraic fact that Clifford multiplication by an ASD form annihilates positive chirality spinors.
\textbf{(Holonomy)} follows from a simple curvature computation, and \textbf{(Isometry)} is more involved, using the Green's operator to construct harmonic spinors on the transformed bundle.

Various generalisations of the Nahm transform (many of which preceded the toric Nahm transform historically) are known; see~\cite{jardimSurveyNahmTransform2004} for a more complete survey. 
One important example is $\real^4/\Lambda \to (\real^4)^* / \Lambda^*$ for $\Lambda$ a subgroup of $\real^4$.
The $\Lambda = \{0\}$ case is related to the ADHM construction~\cite{atiyahConstructionInstantons1978}, while the monopole case $\Lambda = \real$ is the original transform giving Nahm's equations~\cite{nahmAllSelfdualMultimonopoles1983}.

One can also study the Nahm transform on hyperk\"ahler spaces that are not flat.
On asymptotically locally Euclidean (ALE) spaces, Kronheimer--Nakajima~\cite{kronheimerYangMillsInstantonsALE1990} developed a generalised ADHM construction. 
More recently, Cherkis--Larrain-Hubach--Stern developed a Nahm transform on multi-Taub-NUT spaces in a series of papers~\cite{cherkisInstantonsMultiTaubNUTSpaces2021,cherkisInstantonsMultiTaubNUTSpaces2021a,cherkisInstantonsMultiTaubNUTSpaces2025}.

Another direction for generalisation of the Nahm transform is to higher dimensions.
Let $(X^n, \Phi)$ be the data of an oriented Riemannian manifold and a closed $(n-4)$-form $\Phi \in \Omega^{n-4}(X)$.
A vector bundle $E \to X$ with connection $A$ is a \emph{generalised ($\Phi,\lambda$) instanton} if $\star (F_E \wedge \Phi) = \lambda F_E$.
A basic question is whether for any tuple $(X, \Phi, \lambda)$ there exists a dual pair $(\wh{X}, \wh{\Phi}, \wh{\lambda})$ such that the moduli space $\calM(X,\Phi,\lambda)$ of gauge equivalence classes of generalised instantons on $X$ is isometric to $\calM(\wh{X},\wh{\Phi}, \wh{\lambda})$.
In certain special cases, the answer is known to be affirmative---for example, on hyperk\"ahler tori, where the dual manifold is the usual dual torus.
The proof in this setting is a direct generalisation of the proof for the 4-dimensional toric Nahm transform (see~\cite{bartocciFourierMukaiNahmTransforms2009}).

Over the past 30 years, much effort has been put into the study of the generalised instanton problem in the special holonomies $G_2$ and $\Spin(7)$; Lewis~\cite{lewisSpin7Instantons} constructed the first examples of $\Spin(7)$ instantons on the compact irreducible $\Spin(7)$ manifold of Joyce obtained by resolving singularities of the quotient of an $8$-torus by a group of order 16.
The transversality theory and virtual dimension count for the moduli space of $\Spin(7)$ instantons were carried out by Mu\~noz--Shahbazi~\cite{munozTransversalityModuliSpace2018}.
Recently, various examples of $\Spin(7)$ instantons were constructed by Alonso--Madnick--Windes~\cite{alonsoGT2025}.

The focus of this paper is the Nahm transform on an 8-dimensional torus $T^8$ equipped with a Cayley form $\Phi$: in coordinates $(x^1, \ldots, x^8)$,
\begin{align}
    \Phi &= \di x^{1234} + \di x^{1256} - \di x^{1278} + \di x^{1357} + \di x^{1368} + \di x^{1458} - \di x^{1467}  \nonumber \\
         &\qquad - \di x^{2358} + \di x^{2367}  + \di x^{2457} + \di x^{2468}  - \di x^{3456} + \di x^{3478}  + \di x^{5678}. \label{eq:cayley}
\end{align}
As a map $\Lambda^2(\real^8)^* \to \Lambda^2(\real^8)^*$, the operator $\star (\cdot \wedge \Phi)$ has eigenvalues $1$ and $-3$, of respective multiplicities 21 and 7.
Denote by $\Omega^2_{21}(T^8)$ and $\Omega^2_7(T^8)$ the eigenspaces of $\star (\cdot \wedge \Phi)$ on $\Omega^2(T^8)$ corresponding to the eigenvalues $1$, $-3$ respectively.
A $\Spin(7)$ instanton is a bundle with connection $(E, A)$ such that the curvature $F_E$ satisfies $\star (F_E \wedge \Phi) = F_E$, and therefore $F_E \in \Omega^2_{21}(\ad E)$---that is, a $\Spin(7)$ instanton is defined to be a generalised $(\Phi, 1)$ instanton.

In the $\Spin(7)$ setting, the proofs of \textbf{(Vanishing)} and \textbf{(Holonomy)} used in the hyperk\"ahler setting fail immediately, both for the same reason: Clifford multiplication by the curvature $c(F_E) \colon S^\pm \to S^\pm$ does not satisfy a strong vanishing condition like $c(F_E) S^+ = 0$ in the ASD case.
In fact, \textbf{(Vanishing)} is false in the $\Spin(7)$ setting.
\begin{restatable}{theorem}{NoVanishing}
    \label{thm:van}
    There is no generic vanishing theorem for $\Spin(7)$ instantons: there exist $\Spin(7)$ instantons $E_1, E_2$ over a flat 8-torus such that the Dirac operator of $E_1$ (resp.\ $E_2$) has non-zero $\ker D \subset \Gamma(S^+ \otimes E_1)$ (resp.\ $\Gamma(S^- \otimes E_2)$).
\end{restatable}
Theorem~\ref{thm:van} is proved in Section~\ref{sec:vanishing} by directly constructing examples of such instantons $E_1, E_2$ on a square torus using a $\Spin(7)$ analogue of the Kodaira vanishing theorem.

Given a $\Spin(7)$ instanton line bundle $L$ that is positive (in the sense that $i F_L$ defines a symplectic form whose induced orientation agrees with the orientation on $T^8$) we define in Section~\ref{sec:curvature} an \emph{asymptotic Nahm transform} of the bundle $E(k) := E \otimes L^{\otimes k}$ for $k \gg 0$, denoted $\wh{E(k)}$.
The transform $\wh{E(k)}$ has a canonical connection $\wh{A(k)}$, and we define the \emph{$\ell$th asymptotic holonomy $H'_\ell$} to be the orthogonal complement to the space of constant 2-forms whose contraction with the curvature $\wh{F(k)}$ has, pointwise, norm of order $O(k^{\tfrac12-\ell})$ as $k \to \infty$.
The \emph{asymptotic holonomy} $H' = \lim_{\ell \to \infty} H'_\ell$ is therefore the space of forms to which $\wh{F(k)}$ reduces modulo terms decaying faster than $k^{-\ell}$ for every $\ell$ as $k\to \infty$.

Sections~\ref{sec:pseudo} and~\ref{sec:mellin} develop heat kernel estimates used in Section~\ref{sec:curvature} to approximate the curvature $\wh{F(k)}$.
These approximations are used to prove the following main theorem of the paper.

\begin{theorem}
    Let $(T^8, \omega, L)$ be a principally polarised abelian variety with K\"ahler form $\omega$, line bundle $L$ satisfying $F_L = -i\omega$, and $\Spin(7)$ structure such that $L$ is a $\Spin(7)$ instanton.
    Let $(E,A)$ be a $\Spin(7)$ instanton on $T^8$.
    The asymptotic Nahm transform $(\wh{E(k)}, \wh{A(k)}) \to \wh{T^8}$ of $E(k) := E \otimes L^{\otimes k}$ has the following asymptotic reduction of holonomy as $k \to \infty$:
    \[
        H'_0 \subset \langle \omega \rangle, \qquad
        H'_1 \subset \Lambda^2_{21} (\real^8)^*,
   \]
   where $\omega$ is identified with an element of $\Lambda^2 (\real^8)^*$ by the canonical identification with a 2-form on the dual abelian variety and then using that it is covariant constant.

   Moreover, there exist examples of $\Spin(7)$ instantons for which $H' = H'_2 = \mathfrak{u}(1)^{\oplus 4}$; in particular, the transform is not `asymptotically $\Spin(7)$' for any choice of $\Spin(7)$ structure on the dual torus because the (Lie algebra) rank of $H'$ satisfies $\rk(H') = 4 > \rk(\so(7))  = 3$. 
\end{theorem}

\paragraph*{Acknowledgements} I thank my advisor Mark Stern for posing to me the question of $\Spin(7)$ Nahm transforms, as well as for many helpful discussions on the content of this paper.

\section{Algebraic preliminaries}
\label{sec:algebra}

The Lie algebra $\so(7)$ of $\Spin(7)$ is a 21-dimensional subalgebra of the Lie algebra of skew-symmetric $8 \times 8$ real matrices.
Fix a choice of $\gamma$ matrices for the $\Spin(7)$ structure: anticommuting matrices $\gamma^1, \ldots, \gamma^7 \in \so(8)$ pairwise orthogonal with respect to the inner product $\langle A, B \rangle = \Tr(A^t B)$ with the property that $(\gamma^i)^2 = -I$ and $\gamma^{jk} := \frac{1}{2} [\gamma^j, \gamma^k] = \gamma^j \gamma^k$ are an orthogonal basis for $\so(7)$ for $1 \le j < k \le 7$.\footnote{With an identification $\real^8 = \octo$ fixed, one could take $\gamma^j$ to be the right multiplication matrix by the $j$th imaginary octonion, for example.}
The Cayley form is $\Phi = -\frac{1}{6} \sum_{j=1}^7 \gamma^j \wedge \gamma^j$.
As skew-symmetric matrices, we can interpret $\gamma^j, \gamma^{jk}$ as 2-forms on $\real^8$, also denoted by $\gamma^j, \gamma^{jk}$ when no confusion will arise.
The space of forms spanned by $(\gamma^j)_{j=1}^7$ is denoted by $\Lambda^2_7$ and the space of forms spanned by $(\gamma^{jk})_{j,k=1, j\neq k}^7$ is denoted by $\Lambda^2_{21}$.

Given any multi-index $I = (i_1,\ldots,i_k)$, let $\gamma^I := \gamma^{i_1} \cdots \gamma^{i_k}$.
Denote by
\begin{equation}
    \label{eq:matdecomp}
    M_{8\times 8}(\real) = M_1 \oplus M_7 \oplus M_{21} \oplus M_{35}
\end{equation}
the decomposition of the algebra of $8 \times 8$ real matrices into irreducible $\so(7)$ representations of dimensions $1, 7, 21, 35$, respectively spanned by $I$, the $\gamma^j$, the $\gamma^{jk}$, and the $\gamma^{jk\ell}$ for $j < k < \ell$. 

The even complexified Clifford algebra of $\real^8$ splits as $\Cl^0(\real^8) \otimes_{\real} \comp = \End_{\comp}(S^+) \oplus \End_{\comp}(S^-)$ for two complex 8-dimensional $\Spin(8)$-representations $S^+, S^-$, the spinor representations.
There is a further decomposition $S^+ = S^+_1 \oplus S^+_7$ into irreducible $\Spin(7)$ representations of respective dimensions $1, 7$.
The following properties of these representations with respect to the action of the Clifford algebra are well known (see~\cite[Ch.~13]{joyceCompactManifoldsSpecial2000}, for example).
\begin{lemma}
    \label{lem:alg:spin7basic}
    All of the following inclusions hold.
    \begin{align*}
        c(\Lambda^2_7) S^+_1 &\subset S^+_7, \\
        c(\Lambda^2_7) S^+_7 &\subset S^+_1, \\
        c(\Lambda^2_{21}) S^+_1 &= 0, \\
        c(\Lambda^2_{21}) S^+_7 &\subset S^+_7.
    \end{align*}
\end{lemma}

Let $e_1, \ldots, e_8$ be an orthonormal basis of $\real^8$.
For a matrix $M = M_{ab} e_a \otimes e_b^*$, define $c(M) = \frac{1}{2} M_{ab} c(e_a) c(e_b)$.
Let $J$ be a complex structure on $\real^8$ with corresponding 2-form $\omega = \langle J\cdot, \cdot\rangle$.
The Clifford multiplication $c(\omega)$ acting on the spinors decomposes $S^+$ into orthogonal eigenspaces with corresponding eigenvalues $-4i, 0, 4i$ of multiplicities $1, 6, 1$ (see~\cite[Lemma~7.10]{bravermanVanishingTheoremsKernel1998}, for example).
Likewise, $S^-$ decomposes into orthogonal $c(\omega)$-eigenspaces with eigenvalues $-2i, 2i$, both of multiplicity 4.
Let $P_{i\mu}$ denote the orthogonal projection onto the $i\mu$-eigenspace of $c(\omega)$.

\begin{lemma}
    \label{lem:alg:seigen}
    For $0 \neq v \in S^+_1$, the set $\{c(\gamma^j)v\}_{j=1}^7$ is an orthogonal basis of $S^+_7$ and $c(\gamma^j)c(\gamma^k)v = -16\delta_{jk} v$.
\end{lemma}
\begin{proof}
    In general, for skew-symmetric matrices $A, B$ it holds that $[c(A), c(B)] = -2c([A,B])$.
    In particular, $[c(\gamma^{jk}), c(\gamma^k)] = 4c(\gamma^j)$ and thus
    \begin{align*}
        \langle c(\gamma^j)v, c(\gamma^k)v \rangle &= \frac{1}{4}\langle [c(\gamma^{jk}), c(\gamma^k)] v, c(\gamma^k) v \rangle \\
                                                   &= \frac{1}{4}\langle c(\gamma^{jk}) c(\gamma^k) v, c(\gamma^k) v \rangle. \tag{$c(\gamma^{jk})v = 0$}
    \end{align*}
    Since $c(\gamma^{jk})$ is skew-adjoint, it follows that $\langle c(\gamma^j)v, c(\gamma^k)v\rangle$ is imaginary.
    On the other hand, we have
    \begin{align*}
        \langle c(\gamma^j)v, c(\gamma^k)v\rangle &= -\langle c(\gamma^k) c(\gamma^j) v, v\rangle \\
                                                  &= -\frac{1}{2} \langle  \{ c(\gamma^k), c(\gamma^j)\} v, v \rangle
    \end{align*}
    is real, since $[c(\gamma^k), c(\gamma^j)] v = 0$ and $\{c(\gamma^k), c(\gamma^j)\}$ is self-adjoint.
    Thus $c(\gamma^j)v, c(\gamma^k)v$ are orthogonal for all $j \neq k$.
That $c(\gamma^k)c(\gamma^j)v = 0$ for $k \neq j$ follows since this quantity is orthogonal to $v$, but $c(\gamma^k)c(\gamma^j)v \in S^+_1$ by Lemma~\ref{lem:alg:spin7basic}.

    For a given $j$, since $\gamma^j$ is a complex structure and $c(\gamma^j)$ acts as 0 on the 6-dimensional space $\vspan \{ c(\gamma^k)v : k \neq j\}$, the orthogonal complement $\vspan\{v, c(\gamma^j)v\}$ is the sum of the $-4i$ and $4i$ eigenspaces of $c(\gamma^j)$, on which $c(\gamma^j)^2 = -16$.
\end{proof}

Let $T \in \End(S)$ be any spinor endomorphism, and let $\omega$ be a 2-form corresponding to a complex structure on $\real^8$.
The projection $P_{-4i}$ onto the $-4i$ eigenspace of $c(\omega)$ on $S$ is rank 1; thus there exists a scalar $\tau(T)$ such that $P_{-4i} T P_{-4i} = \tau(T) P_{-4i}$.
Since the complexified Clifford algebra is the endomorphism algebra of the spinors, to determine $\tau$ it suffices to evaluate $\tau(c(e_I)) := \tau(c(e_{i_1}) \cdots c(e_{i_k}))$ for some multi-index $I = (i_1, \ldots, i_k)$, where $e_1, \ldots, e_8$ is an orthonormal basis of $\real^8$.
For chirality reasons, $\tau(c(e_I)) = 0$ when $|I|$ is odd.
When $|I|$ is even, there is a Wick formula.
\begin{lemma}
    For a matrix $M$,
    \[
        \tau(c(M)) = -\frac{1}{2} \mathrm{Tr}\,M - \frac{i}{2} \langle \omega, M \rangle.
    \]
    If $|I| = 2q, q > 1$, then
    \begin{equation}
        \label{eq:pfaff}
        \tau(c(e_I)) = \sum_\sigma \mathrm{sgn}(\sigma) \prod_{j=1}^q \tau(c(e_{\sigma(2j-1)}) c(e_{\sigma(2j)}));
    \end{equation}
    the sum runs over all permutations $\sigma$ with the property that $\sigma(2j-1) < \sigma(2j)$ for each $j = 1, \ldots, q$, and $\sigma(1) < \sigma(3) < \cdots < \sigma(2q-1)$.
\end{lemma}
\begin{proof}
    Let $v$ be a unit $-4i$ eigenvector for $c(\omega)$.
    Under the adjoint map $\ad(\omega) = [\omega, \cdot]$, there is a Lie algebra decomposition
    \[
        \so(8) = \ker \ad(\omega) \oplus \mathrm{im}\, \ad(\omega).
    \]
    Let $M$ be any skew-symmetric matrix, and compute
    \begin{align*}
        \tau(c([M, \omega])) = \langle c([M, \omega])v, v \rangle = -\frac{1}{2} \langle [c(M), c(\omega)]v, v \rangle = -\frac{1}{2} (-4i \tau(c(M)) + 4i \tau(c(M))) = 0. 
    \end{align*}

    Note that $\ker \ad(\omega) = \mathfrak{u}(4) = \langle \omega \rangle \oplus \mathfrak{su}(4)$.
    For $M \in \ker \ad(\omega)$, $c(M)v = -\frac{1}{4i} c(\omega) c(M) v$; that is, $c(M) v$ is again in the $-4i$ eigenspace of $c(\omega)$.
    It follows that for $M_1, M_2 \in \ker \ad(\omega)$, $[c(M_1), c(M_2)]v = 0$.
    Thus $\tau(c(\mathfrak{su}(4))) = 0$.
    Finally, $\tau(c(\omega)) = -4i$ whence
    \[
    \tau(c(M)) = -\frac{1}{2} \mathrm{Tr}\, M -\frac{i}{2} \langle \omega, M \rangle.
    \]

    The Wick formula is proved by induction.
    Let $I'_r$ denote the sub-index of $(i_1, \ldots, i_{2q})$ in which $i_r, i_{2q}$ are omitted.
    We show that
    \[
        \tau(c(e_I)) = \sum_{r=1}^{2q-1} (-1)^{r-1} \tau(c(e_{i_r}) c(e_{i_{2q}})) \tau(c(e_{I'_r})),
    \]
    from which the claimed Equation~\eqref{eq:pfaff} follows by a standard Pfaffian counting argument.

      For a vector $v \in \real^8$, define
    \[
        c^\pm(v) := \tfrac{1}{2}\big(c(v) \pm i\, c(\omega v)\big), 
    \]
    so $c(v) = c^-(v) + c^+(v)$.
    Note that $[c(\omega), c^\pm(v)] = \pm 2i c^\pm(v)$, and therefore
    \[
        c(\omega) c^\pm(v) = c^\pm(v) ( c(\omega) \pm 2i );
    \]
    that is, $c^\pm(v)$ maps the $i\mu$-eigenspace of $c(\omega)$ into the $i(\mu \pm 2)$-eigenspace.
    In particular, $c^-(v) P_{-4i} = 0$ because $-6i$ is not an eigenvalue of $c(\omega)$.

    We have
    \[
        \{ c(e_a), c^+(e_b) \} = -\delta_{ab} - i \omega_{ab} = \tau(c(e_a) c(e_b)),
    \]
    the last equality following from the above formula for $\tau(c(E_{ab}))$ where $E_{ab}$ is the skew-symmetric matrix with a $1$ in the $(a,b)$ position, a $-1$ in the $(b,a)$ position, and zeroes elsewhere.

    Compute
    \begin{align*}
        P_{-4i} c(e_I) P_{-4i} &= P_{-4i} c(e_{i_1}) \cdots c(e_{i_{2q-1}}) c^+(e_{i_{2q}}) P_{-4i} \\
                               &= P_{-4i} \left(\sum_{r=1}^{2q-1} (-1)^{2q-1-r} \{c(e_{i_r}), c^+(e_{i_{2q}}) \} c(e_{I'_r})\right) P_{-4i} \\
                               &= \sum_{r=1}^{2q-1} (-1)^{r-1} \tau(c(e_{i_r})c(e_{i_{2q}}))P_{-4i} c(e_{I'_r}) P_{-4i} \\
                               &= \sum_{r=1}^{2q-1} (-1)^{r-1} \tau(c(e_{i_r})c(e_{i_{2q}}))\tau(c(e_{I'_r})) P_{-4i}
    \end{align*}
    as desired.
\end{proof}

\begin{lemma}
    \label{lem:innerchain}
    Let $\omega^1, \ldots, \omega^\ell$ be skew-symmetric matrices.
    For a skew-symmetric matrix $\eta = \eta_{ab} e_a \otimes e_b^*$,
    \[
        \frac{1}{2} \eta_{ab} c(e_a) c(\omega^1) \cdots c(\omega^\ell) c(e_b) = \sum_{J} (-2)^{\ell-|J|} c(\omega^{j_1}) \cdots c(\omega^{j_k}) c((\omega^{j'_1} \cdots \omega^{j'_{\ell-k}})^t \eta),
    \]
    where the sum runs over all ordered multi-indices $J = (j_1 < \cdots < j_k), J \subset \{1, \ldots, \ell\}$, and $J' = (j'_1 < \cdots < j'_{\ell-k})$ is the complementary multi-index to $J$.
\end{lemma}
\begin{proof}
    Compute
    \begin{align*}
        \frac{1}{2} \eta_{ab} c(e_a) c(\omega^1) \cdots c(\omega^\ell) c(e_b)  &= \frac{1}{2} \eta_{ab} ([c(e_a), c(\omega^1)] + c(\omega^1) c(e_a)) c(\omega^2) \cdots c(\omega^\ell) c(e_b) \\
                                                                               &= \frac{1}{2} \eta_{ab} (2\omega^1_{sa} c(e_s) + c(\omega^1) c(e_a)) c(\omega^2) \cdots c(\omega^\ell) c(e_b) \\
                                                                            &=  \frac{1}{2}c(\omega^1) \eta_{ab} c(e_a) c(\omega^2) \cdots c(\omega^\ell) c(e_b) \\
                                                                            &\qquad\qquad + 2 \frac{1}{2} (\omega^1 \eta)_{ab} c(e_a) c(\omega^2) \cdots c(\omega^\ell) c(e_b);
    \end{align*}
    the proof follows by induction.
\end{proof}

\begin{definition}
    \label{def:pl}
    Let $\omega$ be the 2-form of a complex structure on $\real^8$ and let $V$ be a subspace of skew-symmetric matrices.
    For $\ell \ge 0$, define
    \[
        \mathcal{P}_\ell(V) := \vspan \left\{ \omega^{\epsilon_0} v_1 \omega^{\epsilon_1} v_2 \cdots v_j \omega^{\epsilon_j} \mid 0 \le j \le \ell, v_i \in V, \epsilon_i \in \{0,1\} \right\} \subset M_{8\times 8}(\real),
    \]
    the span of all matrix products of at most $\ell$ elements of $V$ interleaved with powers of $\omega$. 
\end{definition}
Note that $I, \omega \in \mathcal{P}_0(V)$.

\begin{lemma}
    \label{lem:matrixvan}
    Suppose that a matrix $\eta$ satisfies $\langle \eta, \mathcal{P}_\ell(V) \rangle = 0$.
    Let $r \le s \le t$ be integers, and let $M_j \in V \cup \{\omega\}$ for $1 \le j \le t$.
    If $M_j \in \langle \omega\rangle$ is true for at least $t - \ell$ indices $j$, then
    \[
        \tau (\eta^{ab} c(M_1) \cdots c(M_r) c_a c(M_{r+1}) \cdots c(M_s) c_b c(M_{s+1}) \cdots c(M_t)) = 0.
    \]
\end{lemma}
\begin{proof}
    By Lemma~\ref{lem:innerchain} it suffices to show that for an ordered multi-index $J = (j_1, \ldots, j_k) \subset (r+1, \ldots, s)$ with complementary multi-index $J' = (j'_1, \ldots, j'_{s-r-k})$,
    \[
        \tau (c(M_1) \cdots c(M_r) c(M_{j_1})\cdots c(M_{j_k}) c((M_{j'_1} \cdots M_{j'_{s-r-k}})^t \eta ) c(M_{s+1}) \cdots c(M_t) ) = 0.
    \]
    We show the case where $J'$ is empty; the case for arbitrary $J'$ follows from the fact that if at least $n$ of $M_{j'_1}, \ldots, M_{j'_{s-r-k}}$ are contained in $\langle I, \omega\rangle$, then
    \[
        (M_{j'_1} \cdots M_{j'_{s-r-k}})^t \eta \perp \mathcal{P}_{\ell-(s-r-k-n)}(V).
    \]
    Thus we are to evaluate
    \[
        (M_1)_{i_1 i_2} \cdots (M_t)_{i_{2t-1} i_{2t}} \eta_{ab} \tau(c(e_I))
    \]
    for the multi-index $I = (i_1, i_2, \ldots, i_{2r-1}, i_{2r}, a, b, i_{2r+1}, \ldots, i_{2t-1}, i_{2t})$ (allowing repeated entries).

    Using Equation~\eqref{eq:pfaff}, let $\sigma$ be any admissible permutation of these indices.
    The contribution of the $\sigma$ term is
    \[
        \mathrm{sgn}(\sigma) (M_1)_{i_1 i_2} \cdots (M_t)_{i_{2t-1} i_{2t}} \eta_{ab} \prod_{j=1}^{t+1} (-\delta_{\sigma(2j-1) \sigma(2j)} - i \omega_{\sigma(2j-1) \sigma(2j)}). \tag{$\star$}
    \]
    We show that $(\star)$ vanishes.
    For some $j$, $\sigma(2j-1) = a$.
    If $\sigma(2j) = b$, then $\eta_{ab} (-\delta_{ab} - i \omega_{ab}) = 0$ since $\eta \perp \mathcal{P}_0(V)$.
    Otherwise, there exist indices $k_1, \ldots, k_m$ for which $\sigma(2j) = i_{2k_1}$, $\sigma(2k_i - 1) = i_{2k_{i+1} - 1}$, and $\sigma(2k_m) = b$.
    In this case, we compute up to signs that
    \begin{align*}
        (\star) &= 
        \pm \left( \prod_{j=1, j\neq k_i}^t (M_j)_{i_{2j-1} i_{2j}} (-\delta_{\sigma(2j-1) \sigma(2j)} - i \omega_{\sigma(2j-1) \sigma(2j)})\right) \langle M_{k_m} (I + i\omega) \cdots M_{k_1} (I + i\omega)\eta, I + i\omega \rangle \\
                &= 0,
\end{align*}
because $(I + i\omega) M_{k_m} (I + i\omega) \cdots M_{k_1} (I + i\omega) \in \mathcal{P}_\ell(V)$ by hypothesis and the fact that $\omega^2 = -I$.
\end{proof}
Lemma~\ref{lem:matrixvan} is used in the proof of Lemma~\ref{lem:AHfromVAN}.

\section{Vanishing theorems}
\label{sec:vanishing}

\begin{lemma}
    \label{lem:plusminuscoho}
    With the $\Spin(7)$ structure and orientation given by the Cayley form of Equation~\eqref{eq:cayley}, $H^2_{21}(\real^8/\inte^8, \inte)$ contains cohomology classes $\omega_+, \omega_-$ such that $\omega_+^4 > 0$ and $\omega_-^4 < 0$.
\end{lemma}
\begin{proof}
    For example, with $\omega_+ = \gamma^{12}$, $\omega_- = \gamma^{12} + \gamma^{34} + \gamma^{56}$ and $\vol = \di x^{12345678}$ the volume element,
    \[
        \omega_+^4 = 24 \vol,\qquad \omega_-^4 = -72\vol.\qedhere
    \]
\end{proof}

Lemma~\ref{lem:plusminuscoho} stands in contrast to the 4-dimensional ASD case, where $\omega^2 < 0$ for every ASD form.

Let $L \to X$ be a line bundle on a compact spin even-dimensional manifold.
Say that $L$ is oriented positively (resp.\ negatively) if its curvature $F_L$ is everywhere nondegenerate and the symplectic form $i F_L$ induces the same (resp.\ opposite) orientation to the given one on the manifold.
There is an asymptotic vanishing theorem in the style of the Kodaira vanishing theorem, due to Braverman~{\cite[Theorem~3.2]{bravermanVanishingTheoremsKernel1998}}.
\begin{theorem}[Braverman]
    \label{thm:vanishing}
    Let $L \to X$ be as above, and suppose that $L$ is oriented positively (resp.\ negatively).
    Let $E \to X$ be a Hermitian bundle with compatible connection.
    For $k \gg 0$, the Dirac operator on the bundle $E(k) = E \otimes L^{\otimes k}$ satisfies $\ker D^-_k = 0$ (resp.\ $\ker D^+_k = 0$).
\end{theorem}

As a consequence, there can be no generic vanishing theorem for $\Spin(7)$ instantons.

\NoVanishing*
\begin{proof}
    Let $\omega_+, \omega_-$ be the forms of Lemma~\ref{lem:plusminuscoho}.
    Consider $\real^8/\inte^8$ as a complex manifold in two ways with complex structure $\gamma^{12}$ (resp.\ $\gamma^7$), with respect to which $\omega_+$ (resp.\ $\omega_-$) is $(1,1)$.
    Using systems of multipliers associated to the forms $\omega_+, \omega_-$ (see, for example,~\cite[Theorem~2.6]{beauvilleThetaFunctionsOld}), construct instanton line bundles $L_+, L_-$ with curvatures $-i \omega_+, -i \omega_-$.

    As was shown in Lemma~\ref{lem:plusminuscoho}, the symplectic forms $ \omega_+, \omega_-$ determine different orientations on $\real^8 / \inte^8$.
    Let $E$ be any $\Spin(7)$ instanton and define $\Spin(7)$ instantons $E_1 = E \otimes L_+^{\otimes k}, E_2 = E \otimes L_-^{\otimes k}$.
    For $k \gg 0$, Theorem~\ref{thm:vanishing} gives $\ker D_{A_1} \subset \Gamma(S^+ \otimes E_1)$ and $\ker D_{A_2} \subset \Gamma(S^- \otimes E_2)$. 

    The index of the Dirac operator in each case is
    \[
        \mathrm{ind}\ D = k^4 \frac{\rk(E)}{(2\pi)^4 4!} \int_{T^8} \omega_{\pm}^4 + O(k^3)
    \]
    as $k \to \infty$, so for $k$ large enough $\mathrm{ind}\ D \neq 0$; therefore, these examples are not trivial.
\end{proof}

It is worth further contrasting this situation with that for $\Sp(2)$ and $\SU(4)$ instantons.
Recall that given a hyperk\"ahler structure with complex structures $I, J, K$, an $\Sp(2)$ instanton is an instanton in the generalised sense for the 4-form $\omega_I^2 + \omega_J^2 + \omega_K^2$ where $\omega_I, \omega_J, \omega_K$ are the K\"ahler forms corresponding to the complex structures.
An $\SU(4)$ instanton is one that is primitive Hermitian--Yang--Mills with respect to the complex structure; that is, $F \in \Omega^{1,1}_0(\ad E)$, the traceless $(1,1)$ forms (with respect to the K\"ahler form).
It follows (see~\cite{madnickSpnInstantonsFourierMukaiTransform2024}, for example) that every $\Sp(2)$ instanton is an $\SU(4)$ instanton and every $\SU(4)$ instanton is a $\Spin(7)$ instanton.
Moreover, there is a partial converse to each of these implications due respectively to Verbitsky~\cite{verbitskyHyperholomorphicBundlesHyperKahler1996} and Lewis~\cite{lewisSpin7Instantons}:
\begin{theorem}[Verbitsky, Lewis]
    \label{thm:verlew}
    Let $E$ be a Hermitian bundle on a compact 8-manifold $X$.
    \begin{itemize}
         \item Suppose that $X$ has $\Sp(2)$ structure, that $(E, A)$ is an $\SU(4)$ instanton, and that there is a connection $A_0$ on $E$ so that $(E, A_0)$ is an $\Sp(2)$ instanton.
             Then $(E, A)$ is an $\Sp(2)$ instanton.
        \item Suppose that $X$ has $\SU(4)$ structure, that $(E, A)$ is a $\Spin(7)$ instanton, and that there is a connection $A_0$ on $E$ so that $(E, A_0)$ is an $\SU(4)$ instanton.
             Then $(E, A)$ is an $\SU(4)$ instanton.
    \end{itemize}
\end{theorem}
This theorem does not apply when the base manifold $X$ is not compact; see~\cite{madnickSpnInstantonsFourierMukaiTransform2024} for discussion and counterexamples in the non-compact case.
The proofs of both statements are similar.
To a closed integral 4-form $\Phi$ one associates an integer charge
\begin{equation}
    \label{eq:instantonnumber}
    \kappa_\Phi(F_E) := \frac{1}{8\pi^2} \int_X \Tr(F_E \wedge F_E) \wedge \Phi = \int_X \left(c_2(E) - \frac12 c_1(E)^2\right) \cup [\Phi].
\end{equation}
For a manifold with $\SU(4)$ structure, consider $\Psi := -\frac{1}{2} \omega \wedge \omega + \mathrm{Re}(\Omega)$ for a K\"ahler form $\omega$ and holomorphic volume form $\Omega$.
The $\SU(4)$-invariant eigenspaces of the bilinear form $\kappa_\Psi$ are subspaces of the $\Spin(7)$-invariant eigenspaces of the bilinear form $\kappa_\Phi$, where $\Phi$ is the Cayley form of Equation~\eqref{eq:cayley}, given the inclusion $\SU(4) \subset \Spin(7)$.
Diagonalising $\kappa_\Phi$ yields Theorem~\ref{thm:verlew}; see~\cite[\S 3.2]{lewisSpin7Instantons} for an explicit basis computation.

\section{Pseudodifferential approximations to the heat kernel}
\label{sec:pseudo}

The goal of this section is to construct a pseudodifferential approximation to the kernel of a Dirac operator and prove Lemma~\ref{lem:expsmall}.
The analysis builds on work of Charbonneau--Stern~\cite{charbonneauAsymptoticHodgeTheory2015}, which concerns bundles $E(k) = E \otimes L^{\otimes k}$ over a K\"ahler manifold with an ample holomorphic line bundle $L$.

For the rest of the paper, we work on a (necessarily flat) principally polarised abelian variety $(T^8, L)$ with positive line bundle $L$ of covariant constant curvature $F_L = -i \omega$, $\omega$ the K\"ahler form on $T^8$.
For $k \gg 0$, Theorem~\ref{thm:vanishing} gives $\ker D = \ker D^+ \subset \Gamma(S^+ \otimes E(k))$.
Let $\Pi \colon L^2(T^8, S^+ \otimes E(k)) \to \ker D^+$ denote the orthogonal projection to the Dirac kernel.
Let $\cspec$ denote a constant so that the spectrum of $D^+ D^-$ is contained in $[2k-\cspec, \infty)$ (such a constant is constructed in~\cite[\S 4.5]{bravermanVanishingTheoremsKernel1998}, for example).
It follows that the spectrum of $D^- D^+$ is contained in $\{0\} \cup [2k-\cspec, \infty)$.

Recall that for an operator $B$ with singular values $\lambda_j$, the Hilbert--Schmidt norm of $B$ is defined to be 
\[
    \| B\|_{HS} := \left(\sum_j \lambda_j^2\right)^{1/2}.
\]
The $2k - \cspec$ spectral gap implies that for any $tk \gg 1$, $\Pi$ is approximated well by $e^{-tD^- D^+}$ in the Hilbert--Schmidt norm.
Indeed, we have for some integers $A,B$ that
\begin{align}
    \label{eq:piexp} \| \Pi - e^{-tD^- D^+}\|_{HS} &\le e^{-(2k-\cspec)t/2} \| e^{-tD^- D^+/2}\|_{HS} \\
    \label{eq:epmdecay} \| e^{-tD^+ D^-}\|_{HS} &\in O(t^A k^B e^{-2tk}) \text{ as $tk \to \infty$}
\end{align}
(cf.~\cite[Equation~(2.10)]{charbonneauAsymptoticHodgeTheory2015}).
Henceforth, $D^2$ refers to the positive chirality operator $D^- D^+$ unless otherwise specified.
Equation~\eqref{eq:piexp} justifies approximating $\Pi$ in the Hilbert--Schmidt norm by first constructing an approximate heat kernel for $e^{-tD^2}$.
For such an approximate heat kernel $q_t$ with error
\begin{equation}
    \label{eq:epsilon}
    \epsilon_t := \left(\frac{\partial}{\partial t} + D^2 \right) q_t
\end{equation}
we compute for $Q_t$ the operator with kernel $q_t$ that
\begin{equation}
    \label{eq:expapprox}
    \|e^{-tD^2} - Q_t\|_{HS} \le \int_0^t \|\epsilon_s\|_{HS} \di s;
\end{equation}
thus the problem is construction of a good $q_t$ with $\epsilon_t$ small.

To this end, define the Mehler kernel,\footnote{After a complex rotation, our $\rho(t,k, |x-y|)$ appears as $\sqrt{k} K(\sqrt{k}|x-y|/2, 0; tk/2)$ in~\cite[p. 183]{pauliWaveMechanics}, for example.}
\begin{equation}
    \label{eq:mehler}
    \rho(t,k,|x-y|) := \left( \frac{k}{4\pi \sinh(tk)} \right)^4 e^{-\frac{k|x-y|^2}{4\tanh(tk)}}.
\end{equation}

Since $c(\omega)$ is skew-adjoint as an operator $\Gamma(S^\pm) \to \Gamma(S^\pm)$, restricted to any fibre its distinct eigenvalues are $i \lambda_1, \ldots, i\lambda_r$ for some $\lambda_1 < \cdots < \lambda_r$.
Note that the eigenvalues are constant because $\omega$ is covariant constant.
Let $\Spectrum^+, \Spectrum^- \subset \Spectrum = \{\lambda_1,\ldots,\lambda_r\}$ denote the subsets of the $\lambda_i$ whose corresponding $c(\omega)$-eigenvectors lie in $S^+, S^-$ respectively, and define
\begin{align*}
    \proj_{\pm}(tk) := \sum_{\lambda \in \Spectrum^{\pm}} e^{-\lambda tk} P_{i\lambda}.
\end{align*}
In dimension 8 then (as in Section~\ref{sec:algebra}),
\[
    \Spectrum^+ = \{ -4, 0, 4\}, \qquad \Spectrum^- = \{ -2, 2 \},
\]
whence
\begin{align*}
    \proj_+(tk) &= e^{4tk} P_{-4i} + P_0 + e^{-4tk}P_{4i}, \\
    \proj_-(tk) &= e^{2tk} P_{-2i} + e^{-2tk} P_{2i}.
\end{align*}
The constant $\lambda_* := \min \Spectrum$ plays an outsize role in the analysis to follow. $i\lambda_*$ is the \emph{dominant eigenvalue} for the problem, in the following sense.
On a spin K\"ahler manifold with K\"ahler form $\omega$, there is an identification of spinors with antiholomorphic forms.
Under this identification, the $i\lambda_*$ eigenspace of $c(\omega)$ is $\Omega^{0,0}$, and the Kodaira vanishing theorem proves that if $E$ is a bundle with $(1,1)$ curvature with respect to $\omega$ and $\Pi$ is the projection onto the kernel of the Dirac operator, then $\Pi(I - P_{i\lambda_*}) = 0$.
Thus the $i\lambda_*$ eigenspace is the only relevant one for computing the projection $\Pi$ in this setting.

In the setting where $E$ does not have $(1,1)$ curvature, it is not necessarily true that $\Pi(I - P_{i\lambda_*}) = 0$, or even that $\Pi(I - P_{i\lambda_*})$ is exponentially small relative to $\Pi P_{i\lambda_*}$ as $k \to \infty$.
Still, there is a precise sense (see Lemma~\ref{lem:expsmall}) in which the only terms in an approximation to $\Pi$ whose contribution as $k \to \infty$ are not exponentially small involve somehow the projection $P_{i\lambda_*}$.

In dimension 8, $\lambda_* = -4$, and the projection $P_{-4i}$ onto the $-4i$ eigenspace of $c(\omega)$ is rank 1.

Define
\[
    U_{\pm}(t,k,|x-y|) = \rho(t,k,|x-y|) \proj_{\pm}(tk)
\]
in the ansatz
\begin{equation}
    \label{eq:ansatz}
    q^{\pm}_t = \psi(x,y) U_\pm \sum_{\ell=0}^N u_\ell(x,y)
\end{equation}
for $u_\ell$ to be determined, where $\psi(x,y) \colon (S \otimes E(k))_y \to (S \otimes E(k))_x$ is the parallel transport map and $u_\ell(x,y)$ are sections of $\End(S_y\otimes E(k)_y)$.
Henceforth the $\pm$ decorations are suppressed.

Following~\cite[\S 3.4]{charbonneauAsymptoticHodgeTheory2015}, we work in geodesic normal coordinates about a point $y$.
For this choice of $q$, we obtain
\[
    \epsilon_t = \psi U(L + H) \sum_{\ell=0}^N u_\ell,
\]
where with $J$ the complex structure operator corresponding to the K\"ahler form $\omega$ we have
\begin{equation}
    L := \partial_t + \Delta_E + ikrJ\partial_r + \frac{kr}{\tanh(tk)}\partial_r,
\end{equation}
and
\begin{equation}
    \label{eq:H}
    H = H_h + \psi\inv \sum_{u,v \in \Spectrum} e^{(u-v)tk} P_{iu} c(F_E) P_{iv} \psi,
\end{equation}
a sum of a scalar operator $H_h$ and Clifford multiplication by the curvature weighted by eigenspace.

Next, we construct an approximate inverse to $L$.
\begin{definition}
    \label{def:weightfilt}
    Let $m, c \in \inte$.
    A partial differential operator $Z$ defined on a geodesic ball $B_y$ about $y$ belongs to $W^m_y\langle c\rangle$ if it can be expressed as a finite sum
    \begin{equation}
        \label{eq:weightfilt}
        Z = \sum_{\substack{2p+|I|-|J| \le m \\ c' \le c}} k^p\, (x-y)^J\, e^{c'tk}\, a_{I,J,p,c'}(x,y,tk)\, \frac{\partial}{\partial x^I},
    \end{equation}
    where each coefficient $a = a_{I,J,p,c'}(x,y,s) \in \End (S \otimes E(k))_y$ is a smooth function of $x,y,s$ satisfying for some $C, N$ that
    \[
        |a(x,y,s)| \le C(1+s)^N \qquad \text{for all $s \ge 0$, where $N = 0$ for $c' = c$}.
    \]
    Write $W^m_y := W^m_y\langle 0\rangle$, and let $W^{m,0}_y\langle c\rangle$ denote the subclass of operators of order zero.
\end{definition}
Analogous to the similar filtration of~\cite[\S 3.5]{charbonneauAsymptoticHodgeTheory2015}, the above definition allows for $a$ to grow at most polynomially quickly in $s$ at sub-leading exponential rates $c' < c$, while $a$ must be bounded at the leading rate $c' = c$.

\begin{remark}
    \label{rmk:weightremarks}

    $W$ is a filtration on partial differential operators.
    For all $m, m', c, c'$ it holds that
    \[
        W^m_y\langle c \rangle \circ W^{m'}_y \langle c'\rangle \subset W^{m+m'}_y \langle c+c'\rangle.
    \]
    In particular, $e^{ctk} W^m_y\langle c' \rangle \subset W^m_y\langle c+c'\rangle$.
    Multiplication by $k$ has weight $+2$ (in the sense that $k W^m_y\langle c\rangle \subset W^{m+2}_y \langle c \rangle$), multiplication by $t$ has weight $-2$, multiplication by $(x-y)^j$ has weight $-1$, and $\frac{\partial}{\partial x^i}$ has weight $+1$.
\end{remark}

For multi-indices $J, K$ define
\[
    \mu_{JK}(s) := \sinh(s)^{|J|+|K|}e^{(|K|-|J|)s}.
\]
Compute on a polynomial $e^{ctk} a_{JK}(tk) z^J \overline{z}^K$ about $y$ that
\begin{equation}
    L(e^{ctk}a_{JK}(tk) z^J \overline{z}^K) = \mu_{JK}(tk)\inv (\partial_t + \Delta_E) (\mu_{JK}(tk) e^{ctk} a_{JK}(tk) z^J \overline{z}^K).
\end{equation}
Define
\begin{equation}
    \label{eq:heatop}
    L\inv(e^{ctk} a_{JK}(tk) z^J \overline{z}^K) := \mu_{JK}(tk)\inv \int_0^t \int_{\real^8} \frac{e^{-\frac{|z-w|^2}{4(t-s)}}}{(4\pi(t-s))^4} (\mu_{JK}(sk) e^{csk}a_{JK}(sk) w^J \overline{w}^K) \di w\ \di s. 
\end{equation}

Fix an integer $\TD > 0$.
For any smooth function $A$ of $z, \overline{z}, tk$, define $L\inv A := L\inv p_{2\TD}$, where $p_{2\TD}$ is the degree $2\TD$ Taylor polynomial of $A$ about $y$.
\begin{lemma}[{Equation~(3.32),~\cite{charbonneauAsymptoticHodgeTheory2015}}]
    \label{lem:approxl}
    $LL\inv - I \in W^{-2\TD}_y$.
\end{lemma}
Lemma~\ref{lem:approxl} justifies the introduction of $L\inv$ as an approximate inverse to $L$.

\begin{lemma}
    \label{lem:weightbasic}
    For $Z \in W^{\ell,0}_y\langle c \rangle$,
    \[
        L\inv Z \in W^{\ell-2,0}_y\langle \max\{c,0\}\rangle.
    \]
\end{lemma}
\begin{proof}
    That $L\inv$ lowers weight by 2 follows from Equation~\eqref{eq:heatop} and Remark~\ref{rmk:weightremarks}.
    The content of Lemma~\ref{lem:weightbasic} is the statement on exponential growth rates.

    It suffices by linearity to consider a single term of the Taylor expansion, $p(z, tk) = e^{ctk} a(tk) z^J \overline{z}^K$.
    Compute
    \begin{equation}
        \label{eq:linvp}
        L\inv p =  \int_0^t \frac{\mu_{JK}(sk)}{\mu_{JK}(tk)} e^{csk} a(sk) \int_{\real^8} \frac{e^{-\frac{|z-w|^2}{4(t-s)}}}{(4\pi(t-s))^4}  w^J \overline{w}^K \di w \di s.
    \end{equation}
    By the Wick formula (cf.~\cite[Theorem 3.2.5]{landoGraphsSurfacesTheir2004}), the $w$ integral is a polynomial in $z, t-s$ of degree at most $\frac{|J|+|K|}{2} \le \TD$ in $t-s$.

    Note that
    \begin{align*}
        \mu_{JK}(sk) &= e^{2|K|sk} \left(\frac{1 - e^{-2sk}}{2}\right)^{|J|+|K|} =  2^{-(|J|+|K|)} e^{2|K|sk} \sum_j (-1)^j \binom{|J|+|K|}{j} e^{-2jsk}.
    \end{align*}
    We also have a generalised binomial series for $\mu_{JK}(tk)\inv$:
    \begin{align*}
        \mu_{JK}(tk)\inv &= 2^{|J|+|K|} e^{-2|K|tk} (1-e^{-2tk})^{-|J|-|K|}= 2^{|J|+|K|} e^{-2|K|tk} \sum_{j' \ge 0} \binom{|J|+|K|+j'-1}{j'} e^{-2j'tk}.
    \end{align*}
    Thus
    \begin{equation}
        \label{eq:expandedmu}
        \frac{\mu_{JK}(sk)}{\mu_{JK}(tk)} = 2^{|J|+|K|} e^{-2|K|(t-s)k} \sum_{j, j' \ge 0}(-1)^j  \binom{|J|+|K|}{j} \binom{|J|+|K|+j'-1}{j'} e^{-2j'tk-2jsk}.
    \end{equation}
    For arbitrary $a(s)$ bounded by a polynomial in $s$ (i.e.\ with $|a(s)| \le C (1+s)^N$ for some $C,N$ and for all $s \ge 0$), compute
    \[
        \left|\int_0^t e^{Bs} a(s) \di s\right| \le C\int_0^t e^{Bs} (1+s)^N \di s \le C'\begin{cases}(1+t)^{N+1} & \text{if $B = 0$} \\ (1+t)^N & \text{if $B < 0$} \\ e^{Bt} (1+t)^N & \text{if $B > 0$}. \end{cases}
    \]
    Apply with $B = c+2(|K|-j)$:
    \begin{equation}
        \label{eq:polybound}
        e^{-2(|K|+j')tk} \left|\int_0^t e^{(c+2(|K|-j))sk} a(sk) \di s\right| \le C'\begin{cases} e^{-2(|K|+j')tk} (1+tk)^{N+1} &\text{if $B = 0$} \\ e^{-2(|K|+j')tk} (1+tk)^{N} &\text{if $B < 0$} \\ e^{(c-2j-2j')tk} (1+tk)^N & \text{if $B>0$} \end{cases}.
    \end{equation}
    Note that for all possible $|K|, j, j', c$ it holds that
    \[
        -2(|K|+j'),c-2j-2j' \le \max\{c,0\}.
    \]
    Substitute Equation~\eqref{eq:expandedmu} into Equation~\eqref{eq:linvp}, and use Equation~\eqref{eq:polybound} to bound each term by $e^{\max\{c,0\}tk} q(tk)$ for some $q(tk)$ bounded by a polynomial in $tk$.
\end{proof}

\begin{corollary}
    \label{cor:evstring}
    Let $Z_i \in W^{m_i, 0}_y \langle c_i \rangle$ for $i = 1, \ldots, \ell$.
    Then
    \[
        L\inv Z_1 L\inv Z_2 \cdots L\inv Z_\ell \in W^{\sum m_i - 2\ell, 0}_y\langle c \rangle, 
    \]
    where
    \[
        c := \max_{0 \le j \le \ell} \sum_{i=1}^j c_i
    \]
    denotes the largest sum of any (possibly empty) prefix of the sequence $(c_1, \ldots, c_\ell)$.\footnote{That is, one of the subsequences $\{\}, \{c_1\}, \{c_1, c_2\}, \ldots$}
\end{corollary}
\begin{proof}
    By Remark~\ref{rmk:weightremarks} and Lemma~\ref{lem:weightbasic} applied to $Z_\ell$,
    \[
        Z_{\ell-1} L\inv Z_\ell \in W^{m_{\ell-1} + m_\ell - 2, 0}_y \langle c_{\ell-1} + \max\{c_\ell,0\} \rangle.
    \]
    Applying Lemma~\ref{lem:weightbasic} again,
    \[
        L\inv Z_{\ell-1} L\inv Z_\ell \in W^{m_{\ell-1} + m_\ell - 2 \cdot 2, 0}_y \langle \max\{c_{\ell-1} + \max\{c_\ell,0\}, 0\} \rangle.
    \]
    Observe that $\max\{c_{\ell-1} + \max\{c_\ell,0\},0\} = \max\{c_{\ell-1} + c_\ell, c_{\ell-1}, 0\}$; the proof is finished by induction.
\end{proof}

Fix $\mu_1, \ldots, \mu_{\ell+1} \in \Spectrum$ and for $F_j \in W^{0,0}_y$ define $Z_j = e^{(\mu_j - \mu_{j+1})tk} P_{i \mu_{j}} F_j P_{i \mu_{j+1}}$.
The sum $\sum_{j=1}^\ell (\mu_j - \mu_{j+1}) = \mu_1 - \mu_{\ell+1}$ telescopes, so Corollary~\ref{cor:evstring} yields
\begin{equation}
    \label{eq:decay}
    e^{-\mu_1 tk} L\inv Z_1 \cdots L\inv Z_\ell \in W^{-2\ell,0}_y \langle - \min_j \mu_j \rangle.
\end{equation}

\begin{lemma}
    \label{lem:mineigen}
    Let $Z_j \in W^{0,0}_y \langle \mu_j - \mu_{j+1} \rangle$ be as above and suppose that the coefficients $a^j_{I,J,p,c'}$ of $Z_j$ (cf.~Equation~\eqref{eq:weightfilt}) are all $O((1+s)^N)$ as $s \to \infty$ for some fixed $N$.
    There exists a constant $C$ so that 
    \[
        \int_{T^8 \times T^8} | \psi U_+(x,y) (L\inv Z_1 \cdots L\inv Z_\ell)(x,y)|^2 \di x\, \di y \le C \begin{cases} k^{4-2\ell} (1+tk)^N e^{-2(4+\min_j \mu_j) tk} & t \ge k\inv \\ 2^N t^{-(4-2\ell)} & t < k\inv \end{cases}.
    \]
\end{lemma}
\begin{proof}
    The estimate is a minor modification of Proposition~3.7 of~\cite{charbonneauAsymptoticHodgeTheory2015}, where the interested reader may find a detailed proof.
    Since $\proj P_{i\mu_1} = e^{-\mu_1 tk} P_{i\mu_1}$,
    \begin{align*}
        \psi U_+(x,y) (L\inv Z_1 \cdots L\inv Z_\ell) &\in \rho(t,k,|x-y|) \psi W^{-2\ell,0}_y \langle -\min_j \mu_j \rangle. 
    \end{align*}
    Supposing that $t \ge k\inv$, approximate
    \[
        \rho(t,k,|x-y|) = \left(\frac{k}{4\pi\sinh(tk)}\right)^4 e^{-\frac{kr^2}{4\tanh(tk)}} \approx \frac{k^4}{(2\pi)^4}e^{-kr^2/4} e^{-4tk}
    \]
    and the computation becomes a standard Gaussian integral.
    When $t < k\inv$ (that is, $tk < 1$), we have $(1 + tk)^N \le 2^N$ and absorb $e^{-2(4+\min_j \mu_j) tk}$ into the constant $C$.
\end{proof}

We now define the $u_\ell$ in Equation~\eqref{eq:ansatz} inductively as $u_0 = I, u_\ell = -L\inv H u_{\ell-1}$.
The following estimates are minor modifications of Propositions~3.7, 3.9 and 3.10 of~\cite{charbonneauAsymptoticHodgeTheory2015}.
Note that unlike Lemma~\ref{lem:mineigen}, which applies only to the estimate for the $D^- D^+$ operator, these estimates are valid for both choices of chirality.

\begin{lemma}
    \label{lem:hswt}
    Let $Z \in W^{-2\ell,0}_y$.
    There exists a constant $C$ such that for $k \gg 1$,
    \[
        \| \psi U Z\|_{HS} \le C\begin{cases} k^{2-\ell} & \text{if $t \ge k\inv$} \\ t^{\ell-2} &\text{if $t < k\inv$} \end{cases}, \qquad \|\psi U Z \|_{\sup} \le C \begin{cases} k^{-\ell} & \text{if $t \ge k\inv$} \\ t^{\ell} & \text{if $t < k\inv$} \end{cases}.
    \]
\end{lemma}
\begin{proof}
    We perform the operator norm computation using the Schur test in the case that $t \ge k\inv$.
    Let $z(x,y)$ denote the kernel of $Z$.
    \begin{align*}
        \|\psi U Z\|_{\sup} &\le C_1\sup_{x \in \real^8} \left|\int \frac{k^4}{\sinh(tk)^4} e^{-kr^2/4}(e^{4tk} P_{-4i} + P_0 + e^{-4tk} P_{4i} ) z(x,y) \di y\right| \\
                            &\le C_2\sup_{x \in \real^8} \left|\int k^4 e^{-kr^2/4}(P_{-4i} + e^{-4tk} P_0 + e^{-8tk} P_{4i} ) z(x,y) \di y\right| \\
                            &\le C_3\int k^4 e^{-kr^2/4} \sum_{2p - |J| \le -2\ell} k^{p} r^{|J|} \di y \\
                            &\le C_4 k^{-\ell}.
    \end{align*}
    The $t < k\inv$ case is analogous, where instead $e^{4tk}$ can be absorbed into the constant and $\sinh(tk) \approx tk$.
    The Hilbert--Schmidt case is Proposition~3.7 of~\cite{charbonneauAsymptoticHodgeTheory2015}.
\end{proof}

\begin{corollary}
    \label{cor:epshs}
    With $\epsilon^N_t$ the error as in Equation~\eqref{eq:epsilon} for the $N$th order approximation $q^N_t = \psi U \sum_{\ell=0}^N u_\ell$, there exists a constant $C$ such that
    \[
        \| \epsilon^N_{t} \|_{HS} \le C \begin{cases} k^{2 - N} & \text{if $t \ge k\inv$} \\ t^{N-2} & \text{if $t < k\inv$} \end{cases}.
    \]
\end{corollary}

To construct approximations to $\Pi$, take $t = k^{-\alpha}$ for some $0 < \alpha < 1$; henceforth the arbitrary choice $\alpha = \frac{1}{2}$ is used.
\begin{proposition}
    \label{prop:goodapp}
    There is a constant $C$ so that for $k \gg 0$,
    \[
        \| \Pi - Q^N_{k^{-1/2}}\|_{HS} \le C k^{\frac{3}{2}-N}.
    \]
\end{proposition}
\begin{proof}
    Apply Corollary~\ref{cor:epshs} together with Equations~\eqref{eq:piexp} and~\eqref{eq:expapprox}.
\end{proof}

For $D^- D^+$ there is a sharper estimate.
Split $H$ of Equation~\eqref{eq:H} as $H = H^S + H^D$ where $H^S = (I-P_{i\lambda_*})H(I-P_{i\lambda_*})$, $H^D = H - H^S$ into a small part and a dominant part.
Lemma~\ref{lem:mineigen} with each $Z_i = H^S$ and for $t \ge k\inv$ gives the following.
\begin{lemma}
    \label{lem:expsmall}
    There exists a constant $C$ such that as $tk \to \infty$,
\begin{equation}
    \label{eq:small}
    \int_{T^8 \times T^8} |\psi U_+(x,y) (L\inv H^S)^\ell(x,y)|^2\di x\, \di y \le Ck^{4-2\ell} (1 + tk)^N e^{-8tk}.
\end{equation}
In particular, this quantity decays exponentially fast as $tk \to \infty$.
\end{lemma}
\begin{proof}
Indeed,
\[
    H^S = \sum_{\mu_1, \mu_2 \in \Spectrum^+ \setminus \{\lambda_*\}} P_{i\mu_1} H P_{i\mu_2};
\]
since $\Spectrum^+ \setminus \{\lambda_*\} = \{0,4\}$ it holds that $-2(4 + \min_j \mu_j) \le -8$ in each application of Lemma~\ref{lem:mineigen}.
\end{proof}

Thus terms contributing to the approximation $q^+_\ell$ either contain $H^D$ (and thus a rank-1 spinor projection $P_{i\lambda_*}$) or else are exponentially small as $tk \to \infty$ in Hilbert--Schmidt norm.

\section{Estimates for the Green's operator}
\label{sec:mellin}

Let $G = (D^+ D^-)\inv$ denote the Green's operator.
Consider the following approximation to $G$:
\[
    R^N := \int_0^{k^{-1/2}} Q^{N,-}_t \di t,
\]
where $Q^{N,-}_t$ is the $N$th order approximate heat kernel constructed in the previous section for the ansatz $q^-_t$, $Q^{N,-}_t \colon L^2(S^- \otimes E \otimes L^{\otimes k}) \to L^2(S^- \otimes E \otimes L^{\otimes k})$.
Compute
\begin{align*}
    D^+ D^- R^N &= I - Q^{N,-}_{k^{-1/2}} + \int_0^{k^{-1/2}} \epsilon^{N,-}_t \di t,
\end{align*}
where $\epsilon^{N,-}_t = (\partial_t + D^+D^-)Q^{N,-}_t$.

From Corollary~\ref{cor:epshs} and Equation~\eqref{eq:expapprox},
\[
    \| Q^{N,-}_{k^{-1/2}}\|_{HS} \le \| e^{-k^{-1/2} D^+ D^-} \|_{HS} + \|Q_{k^{-1/2}}^{N,-} - e^{-k^{-1/2} D^+ D^-}\|_{HS} \in O(k^{\frac{3}{2} - N})
\]
as $k \to \infty$, where the exponential term is absorbed by Equation~\eqref{eq:epmdecay}.
Thus
\begin{align*}
    \| D^+ D^-(R^N - G) \|_{HS} &\le \| Q^{N,-}_{k^{-1/2}}\|_{HS} + \int_0^{k^{-1/2}} \|\epsilon^{N,-}_t\|_{HS} \di t \\
                                &\in O(k^{\frac{3}{2} - N})
\end{align*}
as $k \to \infty$.

Since $D^+ D^- \ge 2k - \cspec$, for $k \gg 1$ it follows that $G - R^N$ is Hilbert--Schmidt and
\begin{equation}
    \label{eq:greensest}
    \| G - R^N \|_{HS} \in O(k^{\frac{1}{2}-N}).
\end{equation}
\section{Asymptotic curvature and holonomy}
\label{sec:curvature}

Let $(T^8, \omega, L)$ be a principally polarised abelian variety, and let $\wh{T^8}$ be the dual abelian variety parametrising flat line bundles over $T^8$.
For $t \in \wh{T^8}$ let $P_t \to T^8$ denote the corresponding line bundle.
Throughout, we use the canonical isomorphism $T T^8 \simeq T^* \wh{T^8}$ and metric duality to identify forms on $T^8$ with forms on $\wh{T^8}$.

By Theorem~\ref{thm:vanishing}, for $k \gg 0$ there is a Nahm transform of the bundle $(E(k), A(k))$ on $T^8$ to the bundle $(\wh{E(k)}, \wh{A(k)})$ on $\wh{T^8}$ defined by $\wh{E(k)}_t = \ker D^+_{k,t}$, with $D_{k,t}$ the Dirac operator on the bundle $S \otimes E(k) \otimes P_t = S \otimes E \otimes L^{\otimes k} \otimes P_t$.
Thus
\[
    D_{k,t} = D_{k,0} + i t^a c(\di x^a),
\]
where $t^a$ are coordinates on $\wh{T^8}$ and $x^a$ coordinates on $T^8$.

Since $\wh{E(k)}$ is a subbundle of the trivial bundle over $\wh{T^8}$ with fibres $L^2(S^+ \otimes E(k))$ and trivial connection $\di$, it inherits a canonical connection $\wh{A(k)} = \Pi_k \di \Pi_k$.
Let $G_k = (D_k^+ D_k^-)\inv$ be the Green's operator for the Dirac Laplacian.
From the formula $\Pi_k = I - D_k G_k D_k$ for the projection together with $\Pi_k D_k = D_k \Pi_k = 0$ we compute
\begin{equation}
    \label{eq:curv}
    \wh{F}_{k,ab}(t) = -\Pi_k(t) c(\di x^a) G_k(t) c(\di x^b) \Pi_k(t) + \Pi_k(t) c(\di x^b) G_k(t) c(\di x^a) \Pi_k(t).
\end{equation}
Henceforth $k$ is suppressed and the shorthand $c_a := c(\di x^a)$ is used.

A basic question is whether $(\wh{E}, \wh{A})$ has reduced holonomy---that is, is there an orthogonal decomposition of the bundle of 2-forms $\Lambda^2T^*\wh{T^8} = \Lambda^2_A \oplus \Lambda^2_B$ such that $\wh{F} \in \Gamma(\Lambda^2_A \otimes \End \wh{E})$?
For example, the 4-dimensional ASD condition is when $\Lambda^2_A$ consists of the ASD forms and $\Lambda^2_B$ the self-dual forms, and the $\Spin(7)$ instanton condition is when $\Lambda^2_A = \Lambda^2_{21}$ and $\Lambda^2_B = \Lambda^2_7$.

We define an appropriate notion of asymptotic reduced holonomy.
The bundle of 2-forms $\Lambda^2 T^* \wh{T^8}$ is trivial.
Fix a trivialisation and so identify the fibres of $\Lambda^2 T^* \wh{T^8}$ with $\Lambda^2 (\real^8)^*$ (that is, $8 \times 8$ skew-symmetric matrices).
Identify $\eta \in \Lambda^2 (\real^8)^*$ with the corresponding constant form on $\wh{T^8}$ and let $\eta \lrcorner \wh{F} := \eta^{ab} \wh{F}_{ab}$ denote the section of $\End \wh{E}$ obtained by contracting along the form indices.
Let
\[
    \|\eta \lrcorner \wh{F}\|^2 := \int_{\wh{T^8}} \| (\eta\lrcorner \wh{F})(t) \|^2_2 \di t
\]
denote the $L^2$ norm, where $\| (\eta \lrcorner \wh{F})(t)\|_2$ is the $L^2$ (Frobenius) norm of the endomorphism of $\wh{E}_t$.
By Equation~\eqref{eq:curv},
\[
    \|\eta \lrcorner \wh{F}\|^2 = \int_{\wh{T^8}} \| \eta^{ab} \Pi(t) c_a G(t) c_b \Pi(t) \|^2_{HS} \di t,
\]
noting that the Hilbert--Schmidt and Frobenius norms agree since $\eta \lrcorner \wh{F}(t)$ is finite rank as an endomorphism of the Hilbert space $L^2(S^+ \otimes E(k))$.

Define 
\[
    H_\ell = \left\{ \eta \in \Lambda^2( \real^8)^*\mid \|\eta \lrcorner \wh{F}(t)\|_2 \in O(k^{\tfrac12-\ell}) \text{ for all $t \in \wh{T^8}$} \right\},
\]
the space of $\eta$ for which $\|\eta \lrcorner \wh{F}(t)\|_2$ decays for all $t$ at a rate $O(k^{\tfrac12-\ell})$ as $k \to \infty$ (and where the constant is allowed to depend upon $\eta$).
Let $H'_\ell$ (the $\ell$th \emph{asymptotic holonomy}) be the orthogonal complement to $H_\ell$.
If $\eta \in H_\ell$, then $\|\eta\lrcorner\wh{F}\| \in O(k^{\tfrac12 -\ell})$ as $k \to \infty$, and so for large $\ell$, $H_\ell$ should be thought of as an approximation to $\Lambda^2_B$.
Likewise, $H'_\ell$ should be thought of as an approximation to $\Lambda^2_A$.

Since $H_j \subseteq H_{j-1}$ these vector spaces converge to a limit space $H = \lim_{j\to \infty} H_j$, and likewise $H' = \lim_{j \to \infty} H'_j$.
We call $H'$ the \emph{asymptotic holonomy} of $\wh{F}$.
We first approximate $\| \eta \lrcorner \wh{F}\|$ using $\Pi \sim Q := Q^N_{k^{-1/2}}$ as constructed in Section~\ref{sec:pseudo} and $G \sim R := R^M$ as constructed in Section~\ref{sec:mellin}.
First compute
\begin{align*}
    \Pi c_a G c_b \Pi - Q^N c_a R^M c_b Q^N &= (\Pi - Q^N) c_a G c_b \Pi + Q^N c_a (G - R^M) c_b \Pi + Q^N c_a R^M c_b (\Pi - Q^N)
\end{align*}
whence, using Hilbert--Schmidt and operator bounds on each individual term,
\[
    \|\eta^{ab}(\Pi c_a G c_b \Pi - Q^N c_a R^M c_b Q^N)\|_{HS} \in O(k^{1/2-N}) + O(k^{3/2-M}) + O(k^{1/2-M}).
\]
Henceforth take $M = N + 1$ for simplicity, so 
\[
    \|\eta^{ab}(\Pi c_a G c_b \Pi - Q^N c_a R^{N+1} c_b Q^N)\|_{HS} \in O(k^{1/2-N}).
\]
$N$ is suppressed in the sequel.
It follows by Cauchy--Schwarz that for $N \ge 2$, as $k \to \infty$
\[
    \| \eta^{ab} \Pi c_a G c_b \Pi \|^2_{HS} \le  \| \eta^{ab} Q c_a R c_b Q \|^2_{HS} + O(k^{\frac12 - N})\| \eta^{ab} Q c_a R c_b Q \|_{HS} + O(k^{1 - 2N}).
\]
Thus
\begin{equation}
    \label{eq:approxcurv}
    \| \eta \lrcorner \wh{F}\|^2 \le \int_{\wh{T^8}} \left(\| \eta^{ab} Qc_a R c_b Q\|^2_{HS} + O(k^{\frac12 - N})\| \eta^{ab} Q c_a R c_b Q \|_{HS}\right)\di t + O(k^{1 - 2N}).
\end{equation}

We compute $H'_\ell$ with the goal of computing $H'$.
The following result is the main tool for executing these computations: it provides a sufficient algebraic criterion to determine whether a 2-form is an element of $H_\ell$.

\begin{lemma}
    \label{lem:AHfromVAN}
    Suppose the curvature of $(E, A)$ takes values in $V$, i.e. $F_E \in \Gamma(V \otimes \ad (E))$, for a linear subspace $V \subset \Lambda^2(\real^8)^*$.
    Let $\eta \in \Lambda^2 (\real^8)^*$ and suppose that (cf. Definition~\ref{def:pl})
    \[
        \langle \eta, \mathcal{P}_\ell(V) \rangle = 0
    \]
    with respect to the inner product $\langle A, B\rangle = \Tr(A^t B)$.
    Then $\eta \in H_\ell$.
\end{lemma}

\begin{proof}
    Taking $N$ to be sufficiently large, it suffices by Equation~\eqref{eq:approxcurv} to show
\[
    \int_{\wh{T^8}} \|\eta^{ab} Q c_a R c_b Q \|^2_{HS} \di t \in O(k^{1-2\ell}),
\]
which we do by proving $\|\eta^{ab} Q(t) c_a R(t) c_b Q(t)\|_{HS}  \in O(k^{\tfrac12-\ell})$ for an arbitrary $t$ suppressed in the sequel.

Recalling for example that $Q$ is the integral operator with kernel $\psi U_+ \sum_{j=0}^N u_j$, we define $Q_j$ to be the integral operator with kernel $\psi U u_j$, define $R_j = \int_0^{k^{-1/2}} Q^{-}_{j,t} \di t$, and then define $T_{\ell_1,\ell_2,\ell_3} := \eta^{ab} Q_{\ell_1} c_a R_{\ell_2} c_b Q_{\ell_3}$.
By Lemma~\ref{lem:hswt}, compute
\begin{align*}
    \| T_{\ell_1,\ell_2,\ell_3}\|_{HS} &\le C\|Q_{\ell_1}\|_{\sup} \|\eta^{ab} c_a R_{\ell_2} c_b\|_{\sup} \|Q_{\ell_3}\|_{HS} \\
                                       &\in O(k^{2-(\ell_1+\ell_2+\ell_3+\tfrac12)}).
\end{align*}
In particular, if $\ell_1+\ell_2+\ell_3 > \ell$ then the term $T_{\ell_1,\ell_2,\ell_3}$ contributes at most $O(k^{\tfrac12-\ell})$ to $\|\eta^{ab} Q(t) c_a R(t) c_b Q(t)\|_{HS}$.

We show that if $\ell_1 + \ell_2 + \ell_3 \le \ell$, then $\|T_{\ell_1,\ell_2,\ell_3}\|_{HS}$ decays exponentially fast as $k \to \infty$.
In $T_{\ell_1,\ell_2,\ell_3}$ expand $u_{\ell_1} = -L\inv (H^S + H^D) u_{\ell_1 - 1}$ using the notation of Lemma~\ref{lem:expsmall}.
By that lemma, the term $(L\inv H^S)^{\ell_1}$ has Hilbert--Schmidt norm decaying exponentially fast as $k \to \infty$, so may be ignored.
Likewise, we may discard the term $(L\inv H^S)^{\ell_3}$ from the expansion of $u_{\ell_3}$ in $T_{\ell_1,\ell_2,\ell_3}$.
Each remaining term of $T_{\ell_1,\ell_2,\ell_3}$ is of the form
\[
    \pm \eta^{ab} \psi U_+ L\inv F^1_1 \cdots L\inv F^1_{\ell_1}c_a \left( \int_0^{k^{-1/2}} \psi U_- (-L\inv H)^{\ell_2} \di s \right)c_b \psi U_+ L\inv F^3_1 \cdots L\inv F^3_{\ell_3},
\]
where $F^i_j$ is one of $H^S, H^D$, and not all $F^1_j$ (resp.\ $F^3_j$) are equal to $H^S$.

Suppose that $F^1_1 = H^D$ and $F^3_{\ell_3} = H^D$ (other cases are analogous).
Since $L\inv$ is scalar, the displayed expression acts pointwise on the spinor factor of $S \otimes E(k)$ as a sum of operators of the form
\begin{equation*}
    P_{-4i} \eta^{ab} c(\eta^1_1) P_{i\mu^1_1} \cdots P_{i\mu^1_{\ell_1-1}} c(\eta^1_{\ell_1}) c_a c(\eta^2_1) P_{i\mu^2_1} \cdots c(\eta^2_{\ell_2}) c_b c(\eta^3_1) P_{i\mu^3_1} \cdots c(\eta^3_{\ell_3}) P_{-4i} \tag{$\star$}
\end{equation*}
for some eigenvalues $\mu^1_j, \mu^3_j \in \Spectrum^+, \mu^2_j \in \Spectrum^-$ and for some $\eta^i_j \in \Gamma(V \otimes \ad E(k))$.

Observe that each $P_{i \mu^i_j}$ is a polynomial in $c(\omega)$; for example, $P_0 = \frac{1}{16} c(\omega)^2 + 1$ (acting on $S^+$) and $P_{2i} = \frac{1}{4i} c(\omega) + \frac{1}{2}$ (acting on $S^-$).
Thus any operator $(\star)$ is itself a sum of operators of the form
\[
    P_{-4i} \eta^{ab} c(M_1) \cdots c(M_r) c_a c(M_{r+1}) \cdots c(M_s) c_b c(M_{s+1}) \cdots c(M_t) P_{-4i}
\]
where at least $t - \ell$ of the $M_i$ are contained in $\langle I, \omega\rangle$.
Such operators vanish by Lemma~\ref{lem:matrixvan}.
\end{proof}

\begin{corollary}
    \label{cor:vancor}
    For all $\ell \ge 0$, $H'_\ell\subset \{ A \mid A \in \mathcal{P}_\ell(V), A^t = -A \}$.
\end{corollary}
Note that $\mathcal{P}_0(V) = \vspan\{I, \omega\}$.
As a consequence, $H'_0 \subset \langle \omega \rangle$.
Likewise, we have the following second-order result in the $\Spin(7)$ case.

\begin{theorem}[Second-order asymptotic holonomy for $\Spin(7)$ instantons]
    \label{thm:spin7hol}
    Suppose that $F_E \in \Omega^2_{21}(\ad E)$, and suppose the polarising line bundle $L$ is a $\Spin(7)$ instanton.
    It follows that $H'_1 \subset \Lambda^2_{21} (\real^8)^*$.
\end{theorem}
\begin{proof}
    Since $E(k)$ is a $\Spin(7)$ instanton, we apply Corollary~\ref{cor:vancor} with $V = M_{21}$, the 21-dimensional irreducible representation of $\Spin(7)$.
    Compute $\mathcal{P}_1(M_{21})\subseteq M_1 \oplus M_{21} \oplus M_{35}$, using the decomposition of $8 \times 8$ real matrices as $\Spin(7)$ representations from Equation~\eqref{eq:matdecomp}.
    In particular, the only skew matrices in $\mathcal{P}_1(M_{21})$ are $M_{21}$; i.e., in the span of the $\gamma^{jk}$.
    The theorem follows by Corollary~\ref{cor:vancor}.
\end{proof}

The next question is whether $H' \subset \Lambda^2_{21} (\real^8)^*$---that is, whether the transformed connection is `asymptotically a $\Spin(7)$ instanton'.
One computes $\mathcal{P}_2(M_{21}) = M_{8 \times 8}(\real)$, so no reduction of holonomy can result from Lemma~\ref{lem:AHfromVAN}.
Indeed, there exist examples with $H' = H'_{2} = \mathfrak{u}(1)^{\oplus 4}$ as Lie algebras and therefore $H'$ not contained in $\Lambda^2_{21} (\real^8)^*$ for \emph{any} $\Spin(7)$ structure on the dual torus, as the rank of $\Spin(7)$ is 3.
\begin{example}
    \label{ex:3456}
    Let $L_{\pm}$ be the line bundles on the square torus with first Chern class $\gamma^{34} \pm \gamma^{56}$, and let the twisting bundle for the Nahm transform have first Chern class $\gamma^{12}$.
    For the sum $E = L_- \oplus L_+$ we have $H' = \langle \gamma^{12}, \gamma^{34}, \gamma^{56}, \gamma^7 \rangle$.
\end{example}
\begin{proof}
    Using that $\gamma^{34} \pm \gamma^{56}$ commutes with $\gamma^{12}$ and $(\gamma^{34} \pm \gamma^{56})^\ell = -4(\gamma^{34} \pm \gamma^{56})^{\ell-2}$ for $\ell \ge 3$, we compute
    \begin{align*}
        \mathcal{P}_1(\langle \gamma^{34} \pm \gamma^{56}\rangle) &= \vspan\{I, \gamma^{12}, \gamma^{34} \pm \gamma^{56}, \gamma^{1234} \pm \gamma^{1256}\} \\
        \mathcal{P}_\ell(\langle\gamma^{34} \pm \gamma^{56}\rangle) &= \vspan\{I, \gamma^{12}, \gamma^{34} \pm \gamma^{56}, \gamma^{3456}, \gamma^{1234} \pm \gamma^{1256}, \gamma^{7}\}, \qquad\text{for }\ell \ge 2.
    \end{align*}
    The matrices $\gamma^{ijk\ell} \in M_{35}$ are traceless symmetric.
    Thus $H' \subseteq \vspan \{ \gamma^{12}, \gamma^{34}, \gamma^{56}, \gamma^7 \}$.

    It remains to show that this inclusion is equality.
    This could be checked directly from the definitions of the $u_i$, as many computations simplify given the constant curvature condition and the fact that $L_{\pm}$ has curvature of type $(1,1)$ with respect to the complex structure $\gamma^{12}$.
    Alternatively, we may check (using $\gamma^7$ for example) that in particular $\int_{T^8} \Tr(\gamma^7 \lrcorner \wh{F}) = -2\pi i\langle \gamma^{7}, c_1(\wh{E})\rangle \neq 0$.
    By the index theorem,
    \begin{align*}
        c_1(\wh{E}) &= -\star \frac{1}{6} (k\gamma^{12} + \gamma^{34} \pm \gamma^{56})^3 \\
                    &= - k  (\gamma^{12} \pm \gamma^7) - \frac{k^2}{2} (\gamma^{34} \pm \gamma^{56}) - k^3 \gamma^{12},
    \end{align*}
    where $\star$ is the Hodge star (and again, we identify forms on $T^8$ with forms on $\wh{T^8}$).
    In particular, $\langle \gamma^7, c_1(\wh{E}) \rangle \neq 0$ as desired.
\end{proof}
\printbibliography

\end{document}